\documentclass[12pt]{article}
\usepackage{amsmath,amsthm,amssymb}

\usepackage{mathrsfs}

\allowdisplaybreaks

\newtheorem{theorem}{Theorem}[section]
\newtheorem{corollary}{Corollary}[section]

\newtheorem{proposition}{Proposition}[section]
\theoremstyle{remark}

\theoremstyle{remark}

\theoremstyle{remark}
\newtheorem{remark}{Remark}[section]

\newcommand{\di}{\partial}

\newcommand{\la}{\langle}
\newcommand{\ra}{\rangle}
\newcommand{\rom}[1]{{\rm #1}}

\begin{document}

\makeatletter\@addtoreset{equation}{section}

\begin{center}{\Large \bf
  Meixner class of  non-commutative generalized stochastic processes with freely independent values II. The generating function
}\end{center}

{\large Marek Bo\.zejko}\\
Instytut Matematyczny, Uniwersytet Wroc{\l}awski, Pl.\ Grunwaldzki 2/4, 50-384 Wroc{\l}aw, Poland\\
e-mail: \texttt{bozejko@math.uni.wroc.pl}\vspace{2mm}

{\large Eugene Lytvynov}\\ Department of Mathematics,
Swansea University, Singleton Park, Swansea SA2 8PP, U.K.\\
e-mail: \texttt{e.lytvynov@swansea.ac.uk}\vspace{2mm}

{\small

\begin{center}
{\bf Abstract}
\end{center}

\noindent  
Let $T$ be an underlying space with a non-atomic measure $\sigma$ on it. 
In [{\it Comm.\ Math.\ Phys.}\ {\bf 292} (2009), 99--129] the Meixner class of non-commutative generalized stochastic processes with freely independent values, $\omega=(\omega(t))_{t\in T}$, was characterized  through the continuity of the corresponding orthogonal polynomials.   In this paper, we derive a generating function for these orthogonal polynomials.  The first question we have to answer is: What should serve as a generating function for a system of  polynomials of infinitely many non-commuting variables?
We construct a class of operator-valued functions $Z=(Z(t))_{t\in T}$ such that $Z(t)$ commutes with $\omega(s)$ for any $s,t\in T$. Then a generating function
can be understood as 
$G(Z,\omega)=\sum_{n=0}^\infty \int_{T^n}P^{(n)}(\omega(t_1),\dots,\omega(t_n))Z(t_1)\dotsm Z(t_n)\sigma(dt_1)\dotsm\sigma(dt_n)$,
 where $P^{(n)}(\omega(t_1),\dots,\omega(t_n))$ is (the kernel of the) $n$-th orthogonal polynomial. We derive an explicit form of $ G(Z,\omega)$, which has a resolvent form and resembles the generating function in the classical case, albeit it involves integrals of non-commuting operators. We finally discuss a related problem of the action of the annihilation operators $\partial _t$, $t\in T$. In contrast to the classical case, we prove that the operators $\di_t$  related to the free Gaussian and Poisson processes have a property of globality. This result is genuinely infinite-dimensional, since in one dimension one loses the notion of globality.

 } \vspace{2mm}

\section{Introduction and preliminaries}
This paper serves as a continuation of our research started in \cite{bl}. We recall that the Meixner class of non-commutative generalized stochastic processes with freely independent values
was characterized in  \cite{bl} through the continuity of the corresponding orthogonal polynomials. The main aim of the present paper is to derive the generating function for these orthogonal polynomials.

Let us first briefly recall some known results on the generating function of  Meixner polynomials, in both the classical and free cases.  Below, when speaking of orthogonal polynomials on the real line, we will always assume that their measure of orthogonality has infinite support and is centered.

  According to e.g.\ \cite{chihara} (see also the original paper \cite{Meixner}), the Meixner class of orthogonal polynomials on $\mathbb R$  consists
of all monic orthogonal polynomials $(P^{(n)})_{n=0}^\infty$ whose (exponential) generating function has the form
$$G(z,x):=\sum_{n=0}^\infty \frac{P^{(n)}(x)}{n!}\,z^n=\exp(x\Psi(z)+\Phi(z))=\sum_{k=0}^\infty\frac1{k!}\, (x\Psi(z)+\Phi(z))^k,$$
where $z$ is from a neighborhood of zero in $\mathbb C$, $\Psi$ and $\Phi$ are analytic functions in a neighborhood of zero such that $\Phi(0)=\Psi(0)=\Psi'(0)=0$. This assumption  automatically implies that
\begin{equation}\label{jhgfuy}\Phi(z)=-C(\Psi(z)),\end{equation}
where
$C(z)$ is the cumulant generating function of the measure of orthogonality, $\mu$:
$$C(z)=\sum_{n=1}^\infty\frac{z^n}{n!}\,C^{(n)},$$
$C^{(n)}$ being the $n$-th cumulant of $\mu$. Recall that
$$C(z)=\log\left(\int_{\mathbb R}e^{zs}\,\mu(ds)\right).$$
Each system of Meixner polynomials is characterized by three parameters $k>0$, $\lambda\in\mathbb R$, and $\eta\ge0$. The corresponding orthogonal polynomials satisfy the recursion relation
\begin{equation}\label{fytftdr} xP^{(n)}(x)=P^{(n+1)}(x)+\lambda nP^{(n)}(x)+(kn+\eta n(n-1))P^{(n-1)}(x),\end{equation}
and  the  generating function  takes the form
\begin{equation}\label{drtyd} G(z,x)=\exp\big(x\Psi_{\lambda,\eta}(z)-kC_{\lambda,\eta}(\Psi_{\lambda,\eta}(z))\big),\end{equation} where the functions $\Psi_{\lambda,\eta}(z)$ and $C_{\lambda,\eta}(z)$ are determined by the parameters $\lambda$ and $\eta$ only. In particular, $C_{\lambda,\eta}(z)$ is the cumulant generating function of the measure of orthogonality corresponding to the parameters $k=1$, $\lambda$ and $\eta$.  In fact, for fixed $\lambda$ and $\eta$, the change of parameter $k$ means a dilation of the underlying space $\mathbb R$. We refer the reader to e.g.\ \cite{Meixner} for an explicit form of  $\Psi_{\lambda,\eta}(z)$ and $C_{\lambda,\eta}(z)$. We also note that these functions continuously depend on their parameters $\lambda$ and $\eta$, see \cite{Rod} for details.

Let us now  outline the infinite dimensional case, see \cite{Ly1,Ly2,Rod} for further details.
Let $T$ be a complete, connected, oriented $C^\infty$
Riemannian manifold and let ${\cal B}(T)$ be the Borel
$\sigma$-algebra on $T$. Let $\sigma$ be a
Radon, non-atomic, non-degenerate measure on $(T,\mathcal B(T))$. (For simplicity, the reader may think of $T$ as $\mathbb R^d$ and of $\sigma$ as the Lebesgue measure). Let $\mathcal D$ denote the space of all real-valued infinitely differentiable functions on $T$ with compact support. We endow $\mathcal D$ with the standard nuclear space topology. Let $\mathcal D'$ denote the dual space of $\mathcal D$ with respect to the center space $L^2(T,\sigma)$.
Thus, $\mathcal D'$ consists of generalized functions (distributions) on $T$. Let $\mathscr C$ denote the cylinder $\sigma$-algebra on $\mathcal D'$, i.e., the minimal $\sigma$-algebra on $\mathcal D'$ with respect to which, for any $\xi\in\mathcal D$, the  mapping $\mathcal D'\ni\omega\mapsto\langle \omega,\xi\rangle\in\mathbb R$  is Borel-measurable.  Here and below, $\langle \cdot,\cdot\rangle$ denotes the  pairing between elements of a given linear topological space and its dual space.

Let $\mu$ be a probability measure on $(\mathcal D',\mathscr C)$ (a generalized stochastic process). The cumulant generating function of $\mu$ is given by
$$ C(\xi)=\log\left(\int_{\mathcal D'}e^{\langle \omega,\xi\rangle}\mu(d\omega)\right),\quad \xi\in\mathcal D.$$
The  Meixner class of generalized stochastic processes with independent values may be identified as follows.
We fix arbitrary smooth functions $\lambda:T\to\mathbb R$ and $\eta:T\to[0,\infty)$, and define a probability measure $\mu$ on $(\mathcal D',\mathcal C)$   whose cumulant generating function is
$$C (\xi)=\int_T C_{\lambda(t),\eta(t)}(\xi(t))\,\sigma(dt), \quad \xi\in\mathcal D.
$$ Here, $C_{\lambda(t),\eta(t)}(\cdot)$ is as in \eqref{drtyd}.

Consider the set of all continuous polynomials on $\mathcal D'$, i.e., functions on $\mathcal D'$ which have  the form
$$F(\omega)=\sum_{i=0}^n\langle \omega^{\otimes i},f^{(i)}\rangle,\quad n\in\mathbb N_0.$$
 Here, for each $i$, $f^{(i)}$ belongs to the $i$-th symmetric tensor  power of $\mathcal D$, i.e., $f^{(i)}\in\mathcal D^{\odot i}$, where $\odot$ denotes symmetric tensor product. Note that $\mathcal D^{\odot i}$ consists of all smooth symmetric functions on $T^i$ with compact support.

 For each $f^{(n)}\in\mathcal D^{\odot n}$, we denote by $ P(f^{(n)})=P(f^{(n)},\omega)$ the orthogonal projection of the monomial $\langle \omega^{\otimes n},f^{(n)}\rangle$ onto the $n$-th chaos, i.e., onto the orthogonal difference in $L^2(\mathcal D',\mu)$ of the closures of the polynomials of order $\le n$ and of order $\le n-1$,  respectively. Then
$P(f^{(n)})$ is a continuous polynomial.
By construction, for any $f^{(n)}\in\mathcal D^{\odot n}$ and $g^{(m)}\in\mathcal D^{\odot m}$ with  $n\ne m$, the polynomials $P(f^{(n)})$ and $P(g^{(m)})$  are orthogonal.
Furthermore, for each $\omega\in\mathcal D'$, one can recursively define $P^{(n)}(\omega)\in\mathcal D^{\prime\,\odot n}$, $n\in\mathbb N$, so that
$$ P(f^{(n)},\omega)=\langle P^{(n)}(\omega),f^{(n)}\rangle.$$
The (exponential) generating function of these  polynomials is defined  by
$$  G(\xi,\omega):=\sum_{n=0}^\infty\frac1{n!}\,
\langle P^{(n)}(\omega),\xi^{\otimes n}\rangle
,$$
where $\xi$ is from a neighborhood of zero in $\mathcal D$. We have:
\begin{equation}\label{uyfrutf}   G(\xi,\omega)=\exp\left[
\big\langle \omega(\cdot),\Psi_{\lambda(\cdot),\eta(\cdot)}(\xi(\cdot))\big\rangle-\int_T C_{\lambda(t),\eta(t)}
(\Psi_{\lambda(t),\eta(t)}(\xi(t)))\,\sigma(dt)
\right], \end{equation}
compare with \eqref{drtyd}. Note that the measure $\sigma$ now plays the role of the parameter $k$ in \eqref{drtyd}.

Below, in the fee case, we will use, for many objects,  the same notations as those used for their counterpart in the classical case. However, it should always be clear from the context which objects are meant.

 Introduced by Anshelevich \cite{a1} and Saitoh, Yoshida \cite{SY}, the free Meixner class of orthogonal polynomials on $\mathbb R$  consists of all monic
orthogonal polynomials $(P^{(n)})_{n=0}^\infty$ on $\mathbb R$    whose (usual) generating function has the form
$$G(z,x):=\sum_{n=0}^\infty P^{(n)}(x)z^n=(1-x\Psi(z)-\Phi(z)))^{-1}=\sum_{k=0}^\infty (x\Psi(z)+\Phi(z))^k,$$
where $z$ is from a neighborhood of zero and $\Phi$ and $\Psi$ satisfy the same conditions as in the classical case. Then the function $\Phi(z)$ automatically takes the form as in \eqref{jhgfuy},
but with
$C(\cdot)$ being the free cumulant generating function of the measure of orthogonality, $\mu$:
$$ C(z):=\sum_{n=1}^\infty z^n C^{(n)},$$
where $C^{(n)}$ is the $n$-th free cumulant of $\mu$, see \cite{a1,a5}.

A system of such  polynomials is also characterized by three parameters $k>0$, $\lambda\in\mathbb R$, $\eta\ge0$
and the polynomials satisfy the recursion relation as in \eqref{fytftdr} but with the factors $n$ and $n-1$ being replaced by $[n]_0$ and $[n-1]_0$,  respectively. Here,
for $q\in\mathbb R$ and $n=0,1,2\dots$, we denote $[n]_q:=(1-q^n)/(1-q)$ and so $[n]_0=0$ for $n=0$ and $=1$ for all $n=1,2,\dots$ Thus, 
\begin{align*}
&P^{(0)}(x)=1,\quad P^{(1)}(x)=x,\\
&xP^{(1)}(x)=P^{(2)}(x)+\lambda P^{(1)}(x)+kP^{(0)}(x),\\
&xP^{(n)}(x)=P^{(n+1)}(x)+\lambda P^{(n)}(x)+(k+\eta)P^{(n-1)}(x),\quad n\ge2.
\end{align*}
Furthermore, the generating function $G(z,x)$ takes
the form as in \eqref{drtyd} but with the the resolvent function  replacing the exponential function. In fact, we have \cite{a1}
\begin{equation}\label{kjfdry}
\Psi_{\lambda,\eta}(z)=\frac{z}{1+\lambda z+\eta z^2},\quad C_{\lambda,\eta}(\Psi_{\lambda,\eta}(z))=\frac{z^2}{1+\lambda z+\eta z^2},
\end{equation}
so that
\begin{equation}\label{rtysdtr}
G(z,x)=\left(
1-x\,\frac{z}{1+\lambda z+\eta z^2}+k\,\frac{z^2}{1+\lambda z+\eta z^2}
\right)^{-1}.
\end{equation}
We also note that the class of orthogonal polynomials which is now called the free Meixner class, was derived in the conditionally free central limit theorem and in the conditionally free Poisson limit theorem in \cite{BLS}.

In \cite{a5} (see also \cite{a2}), Anshelevich introduced and studied multivariate orthogonal polynomials of  non-commuting
variables  
with a resolvent-type generating function.
He, in particular, noticed that  the generating function $G(z,x)$ should be defined
for non-commuting indeterminates $(z_1,\dots,z_k)=z$
(which form coefficients  by the orthogonal polynomials) and non-commuting indeterminates $(x_1,\dots,x_k)=x$ (which are variables of the polynomials), and the $z_i$-variables must commute with the $x_j$-variables for all $i,j=1,\dots,k$.
The generating function is then supposed to have the form
\begin{equation}\label{tyder} G(z,x)=\left(1-
\sum_{i=1}^k x_i\Psi_i(z)-\Phi(z)
\right)^{-1}.\end{equation}
We refer to \cite{a2,a5} for an extension of formula \eqref{jhgfuy} to the multivariate case,  for a   recursion relation satisfied by the corresponding orthogonal polynomials,  for an operator model of these polynomials, and for further related results.

In part 1 of this paper, \cite{bl}, we identified the
Meixner class of  non-commutative generalized stochastic processes $\omega=(\omega(t))_{t\in T}$  as those
\begin{itemize}
\item[a)] which have free independent values;
\item[b)] whose  orthogonal polynomials are continuous in $\omega$.
\end{itemize}
The main aim of the present paper is to derive the generating function for a system of orthogonal polynomials as in b).
However, when discussing a generating function for a system of  polynomials of infinitely many non-commuting variables, the first question we have to answer is: What should serve as a generating function? Developing the idea of \cite{a5}, we  will proceed in this paper   as follows.

Think informally of each  polynomial of $\omega$
as
$\langle P^{(n)}(\omega)
,f^{(n)}\rangle$, where $P^{(n)}(\omega)
$ is an operator-valued distribution on $T^n$ and $f^{(n)}$ is a test function on $T^n$.
We will consider a class of test operator-valued functions on $T$, denoted by $\mathcal Z(T)$.  We assume that, for each $Z\in\mathcal Z(T)$ and $t\in T$, the operator $Z(t)$ commutes with each  polynomial $\langle P^{(n)}(\omega)
,f^{(n)}\rangle$. (However, for $s,t\in T$, $Z(s)$ and $Z(t)$ do not need to commute.) In Section~\ref{itiutr68}, we  give a rigorous meaning to a `dual pairing' $\langle P^{(n)}(\omega),Z^{\circledast n}\rangle$ and define a generating function
$$ G(Z,\omega)=\sum_{n=0}^\infty\langle P^{(n)}(\omega),Z^{\circledast n}\rangle,\quad Z\in\mathcal Z(T).$$
Here $Z^{\circledast n}(t_1,\dots,t_n):=Z(t_1)\dotsm Z(t_n)$ for $(t_1,\dots,t_n)\in T^n$.
We also show  that the generating function $ G(Z,\omega)$ uniquely characterizes the corresponding system of polynomials.   In Section~\ref{fyfr}, we  prove that the generating function of the  Meixner system  is given by
\begin{align} G(Z,\omega)&=\left(\mathbf 1-
\big\langle \omega(\cdot),\Psi_{\lambda(\cdot),\eta(\cdot)}(Z(\cdot))\big\rangle+\int_T C_{\lambda(t),\eta(t)}
(\Psi_{\lambda(t),\eta(t)}(Z(t)))\,\sigma(dt)
\right)^{-1}
\notag\\
&=\left(
\mathbf 1-\left\langle\omega,\,\frac{Z}{\mathbf 1+\lambda Z+\eta Z^2}\right\rangle+\int_T \frac{Z(t)^2}{\mathbf 1+\lambda(t)Z(t)+\eta(t)Z(t)^2}\,\sigma(dt)
\right)^{-1},
\label{ftydr} \end{align}
where $\omega$ is Meixner's non-commutative generalized
stochastic processes with freely independent values corresponding to functions $\lambda$ and $\eta$.  The reader is advised to compare  formula \eqref{ftydr}   with the generating function  in the classical infinite dimensional case, formula \eqref{uyfrutf}, and with the generating function in the finite-dimensional free case, formulas \eqref{rtysdtr} and \eqref{tyder}.

Finally, in Section~\ref{jfdrt}, we discuss a related problem of the  action of the annihilation operator at point $t\in T$, denoted by $\di_t$ in \cite{bl}. Recall that, in the classical infinite-dimensional case, the annihilation operator $\di_t$ can be represented as an analytic function of the Hida--Malliavin derivative $D_t$, more precisely $\di_t=\Psi_{\lambda(t),\eta(t)}^{-1}(D_t)$.
(Recall that $D_t$ is the derivative in the direction of the delta-function $\delta_t$.)
We discuss a free counterpart of this result in  free Gauss--Poisson case, i.e., when $\eta\equiv0$. A striking difference to the classical case is that we represent $\di_t$ not just as a function of   the free derivative $D_t$ in the direction $\delta_t$ (this being impossible), but rather as a function of an operator $D_t\mathbb G$.
More precisely, we show that $\di_t=\Psi_{\lambda(t),0}^{-1}(D_t\mathbb G)$.
 Here $\mathbb G$ is a `global' operator, which is independent of $t$.
In fact,  $\mathbb G$ is a sum of certain integrals of $D_s$ over the whole space $T$.
  It should be stressed that this result is genuinely infinite-dimensional, since in one dimension we lose the notion of `globality'.
We expect that a similar result should also hold in the general case, not necessarily when $\eta\equiv0$, and we hope to return to this problem in our future research.

\section{Generating function: construction and uniqueness of corresponding polynomials}\label{itiutr68}

Just as in \cite{bl}, we will assume that $T$ is a locally compact Polish space.
We denote by $\mathcal B(T)$ the Borel $\sigma$-algebra on $T$, and by  $\mathcal B_0(T)$ the collection of all relatively compact sets from $\mathcal B(T)$. For any fixed $A\in\mathcal B_0(T)$, we will denote by $\mathcal B(A)$ the trace $\sigma$-algebra of $\mathcal B(T)$ on $A$, i.e., $\{B\in\mathcal B(T)\mid B\subset A\}$.

\subsection{Construction of integral of an operator-valued function with respect to an operator-valued measure}\label{tyrdert6}

Let $\mathcal G$ be a real separable Hilbert space, and let $\mathscr L(\mathcal G)$ denote the Banach space of all bounded linear operators in $\mathcal G$. We will call a mapping
$Z:T\to\mathscr L(\mathcal G)$
simple if it has a form
\begin{equation}\label{yfrtuyrcd} Z(t)=\sum_{i=1}^n Z_i\chi_{\Delta_i}(t),\end{equation}
where $Z_1,\dots,Z_n\in\mathscr L(\mathcal G)$,  $\Delta_1,\dots,\Delta_n\in\mathcal B_0(T)$, $n\in\mathbb N$, and $\chi_{\Delta_i}(t)$ denotes the indicator function of the set $\Delta_i$. We  denote by $\mathcal Z(T)$ the set of all mappings $Z:T\to\mathscr L(\mathcal G)$  such that there exists a set $A\in\mathcal B_0(T)$ and a sequence of simple mappings $\{Z_n\}_{n=1}^\infty$ which vanish outside $A$ and satisfy
\begin{equation}\label{dtrsw}\sup_{t\in T}\|Z(t)-Z_n(t)\|_{\mathscr L(\mathcal G)}\to0\quad\text{as }n\to\infty.\end{equation}
Clearly, $\mathcal Z(T)$ is a normed vector space equipped with the norm $$\|Z\|_\infty:=\sup_{t\in T}\|Z(t)\|_{\mathscr L(\mathcal G)}.$$
By construction, the set of all simple mappings forms a dense subspace in $\mathcal Z(T)$.

\begin{remark}
It can be easily shown that any mapping  $Z:T\to\mathscr L(\mathcal G)$ which is  continuous and which vanishes outside a compact set in $T$, belongs to $\mathcal Z(T)$.
\end{remark}

Let $\mathcal H$ be another real, separable Hilbert space.  We consider a mapping
$$\mathcal B_0(T)\ni\Delta\mapsto M(\Delta)\in \mathscr L(\mathcal H).$$
We assume:

\begin{enumerate}

\item[(A1)]  $M(\varnothing)=\mathbf 0$.

\item[(A2)] $M(\cdot)$ admits a decomposition
$$M(\Delta)=U(\Delta)+V(\Delta),\quad \Delta\in\mathcal B_0(T),$$
with $U(\Delta), V(\Delta)\in\mathscr L(\mathcal H)$ being such  that, for any mutually disjoint sets $\Delta_1,\Delta_2\in\mathcal B_0(T)$, we have
$$ \operatorname{Ran} U(\Delta_1)\perp \operatorname{Ran} U(\Delta_2),\quad \operatorname{Ran}V(\Delta_1)^*\perp \operatorname{Ran}V(\Delta_2)^*,$$ where
$\operatorname{Ran} A$ denotes the range of a bounded linear operator $A$, and
the symbol $\perp$ refers to orthogonality in $\mathcal H$.

\item[(A3)] For any $A\in\mathcal B_0(T)$,  any sequence of mutually disjoint sets $\Delta_n\in\mathcal B(A)$, $n\in\mathbb N$, and any $F\in\mathcal H$,
$$U\left(\bigcup_{n=1}^\infty\Delta_n\right)F=\sum_{n=1}^\infty U(\Delta_n)F,\quad V\left(\bigcup_{n=1}^\infty\Delta_n\right)^*F=\sum_{n=1}^\infty V(\Delta_n)^*F,$$
where the series converges in $\mathcal H$.
\end{enumerate}

\begin{remark}The reader will see below that assumptions (A1)--(A3) are sufficient for our purposes. 

\end{remark}

For each $Z\in\mathcal Z(T)$, we will now  identify an integral $\int_T Z\otimes dM$ as 
a bounded linear operator in the Hilbert space  $\mathcal G\otimes\mathcal H$.
We fix any $A\in\mathcal B_0(T)$.
Let $Z$ be a simple mapping as in \eqref{yfrtuyrcd} such that $\Delta_i\subset A$ for all $i=1,\dots,n$. Without loss of generality, we may assume that the sets $\Delta_1,\dots,\Delta_n$ are mutually disjoint.
We  define
$$\int_T Z\otimes dU:=\sum_{i=1}^n Z_i\otimes U(\Delta_i)\in\mathscr L(\mathcal G\otimes\mathcal H).$$
 By (A2),
$$\operatorname{Ran}(Z_i\otimes U(\Delta_i))\perp \operatorname{Ran}(Z_j\otimes U(\Delta_j)), \quad i\ne j,$$
where $\perp$ refers to orthogonality in $\mathcal G\otimes\mathcal H$. Hence, for each $F\in\mathcal G\otimes\mathcal H$,
\begin{align}
\left\|\int_T Z\otimes dU\,F\right\|_{\mathcal G\otimes\mathcal H}^2
&= \sum_{i=1}^n \|Z_i\otimes  U(\Delta_i)F\|_{\mathcal G\otimes\mathcal H}^2\notag\\
&\le \sum_{i=1}^n \|Z_i\|_{\mathscr L(\mathcal G)}^2\|\mathbf 1\otimes U(\Delta_i)F\|^2_{\mathcal G\otimes\mathcal H}\notag\\
&\le \left(\max_{i=1,\dots,n}\|Z_i\|_{\mathscr L(\mathcal G)}^2\right)\sum_{i=1}^n\|\mathbf 1\otimes U(\Delta_i)F\|^2_{\mathcal G\otimes\mathcal H}\notag\\
&=\|Z\|_\infty^2\left\|\mathbf 1\otimes U\left(\bigcup_{i=1}^n\Delta_i\right)F\right\|^2_{\mathcal G\otimes\mathcal H}\notag\\
&\le \|Z\|_\infty^2 \| \mathbf 1\otimes U(A)F\|^2_{\mathcal G\otimes\mathcal H}.\label{trstr}
\end{align}
Note that the latter estimate follows from the inequality
$$\|\mathbf 1\otimes U(A_1)F\|_{\mathcal G\otimes\mathcal H}\le \|\mathbf 1\otimes U(A_2)F\|_{\mathcal G\otimes\mathcal H},\quad A_1,A_2\in\mathcal B_0(T),\ A_1\subset A_2,$$
which, in turn, is a consequence of (A2) and (A3).
Hence, by \eqref{trstr},
\begin{equation}\label{ffuftyu}\left\|\int_T Z\otimes dU\right\|_{\mathscr L(\mathcal G\otimes\mathcal H)}\le \|Z\|_\infty \| U(A)\|_{\mathscr L(\mathcal H)}.
\end{equation}

Let now $Z$ be an arbitrary element of $\mathcal Z(T)$, and let $\{Z_n\}_{n=1}^\infty$ be an approximating sequence of simple mappings as in the definition of $\mathcal Z(T)$. By \eqref{ffuftyu}, for any $m,n\in\mathbb N$,
\begin{align*}
\left\|\int_T Z_n\otimes dU-\int_T Z_m\otimes dU\right\|_{\mathscr L(\mathcal G\otimes\mathcal H)}&=\left\|\int_T (Z_n-Z_m)\otimes dU\right\|_{\mathscr L(\mathcal G\otimes\mathcal H)}\\
&\le \|Z_n-Z_m\|_\infty \| U(A)\|_{\mathscr L(\mathcal H)}.
\end{align*}
Hence, $\left\{\int_T Z_n\otimes dU\right\}_{n=1}^\infty$ is Cauchy sequence in $\mathscr L(\mathcal G\otimes\mathcal H)$, and so it has a limit, which we denote by $\int_T Z\otimes dU$. Clearly, the definition of  $\int_T Z\otimes dU$ does not depend on the choice of approximating sequence of simple mappings.

Note that, if $Z(\cdot)$ belongs to $\mathcal Z(T)$, then also $Z(\cdot)^*$ belongs to $\mathcal Z(T)$.
We can therefore define, for each $Z\in\mathcal Z(T)$,
\begin{equation}\label{dtrste}\int_T Z\otimes dV:=\left(\int_T Z^*\otimes dV^*\right)^*.\end{equation}
Finally, we set
$$\int_T Z\otimes dM:=\int_T Z\otimes dU+\int_T Z\otimes dV.$$
By \eqref{ffuftyu} and \eqref{dtrste},
\begin{equation}\label{hcgdxgtrjdx}
\left\|\int_T Z\otimes dM\right\|_{\mathscr L(\mathcal G\otimes\mathcal H)}\le \|Z\|_\infty \left(\| U(A)\|_{\mathscr L(\mathcal H)}+\| V(A)\|_{\mathscr L(\mathcal H)}\right).
\end{equation}
Thus, we have proved

\begin{proposition}\label{igtuy}
Let $M$ satisfy \rom{(A1)--(A3)}. Then, for each $A\in\mathcal B_0(T)$, there exits a constant $C_1(A)\ge0$, such that, for each $Z\in\mathcal Z(T)$ satisfying $Z(t)=0$ for all $t\not\in A$, we have
$$ \left\|\int_T Z\otimes dM\right\|_{\mathscr L(\mathcal G\otimes\mathcal H)}\le C_1(A)\|Z\|_\infty.$$
\end{proposition}

\begin{remark}
The reader is advised to compare our construction of $\int_T Z\otimes dM$ with constructions of  operator-valued integrals available in \cite{BS,HP,parthasarathy}
\end{remark}

Let us consider the special case where $\mathcal G=\mathbb R$, and so $\mathscr L (\mathcal G)=\mathbb R$. As easily seen, the set $\mathcal Z(T)$ is now the space $B_0(T)$ of all bounded measurable functions $f:T\to \mathbb R$ with compact support. Furthermore, for each $f\in B_0(T)$, the operator $\int_T f\,dM:=\int_T f\otimes dM\in\mathscr L(\mathcal H)$  is characterized by the formula
\begin{equation}\label{xdfszers}\left(\int_Tf\,dM\, F_1,F_2\right)_{\mathcal H}:=
\int_Tf\, dM_{F_1,F_2},\quad F_1,F_2\in\mathcal H.\end{equation}
Here, for any $A\in \mathcal B_0(T)$ and any $F_1,F_2\in\mathcal H$, the mapping
$$\mathcal B(A)\ni\Delta\mapsto M_{F_1,F_2}(\Delta):=(M(\Delta)F_1,F_2)_{\mathcal H}\in\mathbb R$$
is a signed measure on $(A,\mathcal B(A))$. By Proposition~\ref{igtuy}, the total variation of   $M_{F_1,F_2}$ on $A$
satisfies
\begin{equation}\label{tydy}|M_{F_1,F_2}|(A)\le C_1(A)\|F_1\|_{\mathcal H}\,\|F_2\|_{\mathcal H}.\end{equation}

\begin{remark} Assume that $T=\mathbb R$ and $M(\cdot)$ is an orthogonal resolution of the identity in $\mathcal H$, i.e., a projection-valued measure on $(\mathbb R,\mathcal B(\mathbb R))$. Then $M(\cdot)$
clearly satisfies the above assumptions and
$\int_{\mathbb R}f\,dM$ is a usual spectral integral (see e.g.\ \cite{BUS,Yosida}).
\end{remark}

\subsection{Generating function uniquely identifies polynomials}\label{fstrs}

We will now consider a sequence $(M^{(n)})_{n=1}^\infty$ of operator-valued measures on $\mathcal B_0(T^n)$, respectively. Our initial assumptions on each  $M^{(n)}$ will be slightly weaker than those in subsec.~\ref{tyrdert6}.

We assume that, for each $n\in\mathbb N$, we are given a function
$$\mathcal B_0(T^n)\ni\Delta\mapsto M^{(n)}(\Delta)\in\mathscr L(\mathcal H)$$
which satisfies the following assumption:

\begin{enumerate}

\item[(B)] For any $F_1,F_2\in\mathcal H$ and  any $A\in\mathcal B_0(T)$, the mapping
$$\mathcal B_0(A^n)\ni\Delta\mapsto M^{(n)}_{F_1,F_2}(\Delta):=(M^{(n)}(\Delta)F_1,F_2)_{\mathcal H}\in\mathbb R$$
is a signed measure on $(A^n,\mathcal B(A^n))$ whose total variation on $A^n$ satisfies
\begin{equation}\label{tdtyuk}|M^{(n)}_{F_1,F_2}|(A^n)\le C_2(A)^n\,\|F_1\|_{\mathcal H}\,
\|F_2\|_{\mathcal H},\end{equation}
where the constant $C_2(A)$ only depends on $A$, and is independent of $F_1,F_2\in\mathcal H$  and $n\in\mathbb N$.
\end{enumerate}

Analogously to \eqref{xdfszers},  we may then identify, for each $f^{(n)}\in B_0(T^n)$, the integral
$\int_{T^n}f^{(n)}\,dM^{(n)}$ an an element of $\mathscr L(\mathcal H)$. (This operator may be though of as a polynomial  of the $n$-th order.)

For any $Z_1,\dots,Z_n\in\mathcal Z(T)$, we define
$$ (Z_1\circledast Z_2\circledast\dots \circledast Z_n)(t_1,t_2,\dots,t_n):=Z_1(t_1)Z_2(t_2)\dotsm Z_n(t_m),$$
where the right hand side is understood in the sense of the usual product of operators.
Note that, in the case where $\mathcal G=\mathbb R$, for any $f_1,f_2\dots,f_n\in\mathcal Z(T)=B_0(T)$, we evidently have
$$f_1\otimes f_2\otimes\dots\otimes f_n=f_1\circledast f_2\circledast\dots\circledast f_n.$$

For each $Z\in\mathcal Z(T)$, we would like to identify an integral $\int_{T^n} Z^{\circledast n}\otimes d M^{(n)}$
as an element of $\mathscr L(\mathcal G\otimes\mathcal H)$. 
However, we cannot do this under the above assumptions, so we define a four-linear form
\begin{align}
&\int_{T^n} Z^{\circledast n}\otimes dM^{(n)}\,(G_1,F_1,G_2,F_2)\notag\\
&\quad:=\int_{T^n} (Z(t_1)\dotsm Z(t_n)G_1,G_2)_{\mathcal G}\,dM^{(n)}_{F_1,F_2}(t_1,\dots,t_n),\quad G_1,G_2\in\mathcal G,\ F_1,F_2\in\mathcal H.\label{ydyd}
\end{align}
As easily follows from the definition of $\mathcal Z(T)$ and (B), the function
 $$ T^n\ni(t_1,\dots,t_n)\to  (Z(t_1)\dotsm Z(t_n)G_1,G_2)_{\mathcal G}\in\mathbb R$$
 is indeed measurable, the integral in  \eqref{ydyd} is finite, and moreover,
  \begin{equation}\label{tfydy}\left|\int_{T^n} Z^{\circledast n}\otimes dM^{(n)}\,(G_1,F_1,G_2,F_2)\right|\le
  \|Z\|_\infty^n C_2(\operatorname{supp}Z)^n \|G_1\|_{\mathcal G}\,\|G_2\|_{\mathcal G}\,
 \|F_1\|_{\mathcal H}\,
\|F_2\|_{\mathcal H}.\end{equation}
 Here, $\operatorname{supp}Z$
 denotes the support of $Z$. 
 Hence, $\int_{T^n} Z^{\circledast n}\otimes dM^{(n)}$ is a bounded (and so continuous) form.

\begin{remark}\label{rtdtrd}
If there exists an operator $Q^{(n)}\in\mathscr L(\mathcal G\otimes\mathcal H)$ such that
$$ (Q^{(n)} G_1\otimes F_1,G_2\otimes F_2)_{\mathcal G\otimes\mathcal H}= \int_{T^n} Z^{\circledast n}\otimes dM^{(n)}\,(G_1,F_1,G_2,F_2),\quad G_1,G_2\in\mathcal G,\ F_1,F_2\in\mathcal H,$$
then we can identify $\int_{T^n} Z^{\circledast n}\otimes dM^{(n)}$ with the operator $Q^{(n)}$. However, the estimate \eqref{tfydy} is not sufficient for this to hold.

\end{remark}

We define a generating function of $(M^{(n)})_{n=1}^\infty$ as follows.
We set
$$\operatorname{Dom}( G):=\{Z\in\mathcal Z(T):\, \|Z\|_\infty C_2(\operatorname{supp}Z)<1\}. $$
Note that for each $Z\in\mathcal Z(T)$, one can find $\varepsilon>0$ such that, for each $a\in(-\varepsilon,\varepsilon)$, $aZ$ belongs to $\operatorname{Dom}( G)$.
By virtue of \eqref{tfydy}, for each $Z\in\operatorname{Dom}( G)$,
\begin{equation}
 G(Z):=\mathbf1+\sum_{n=1}^\infty \int_{T^n} Z^{\circledast n}\otimes dM^{(n)}
\end{equation}
defines a bounded four-linear form on $\mathcal G\times\mathcal H\times \mathcal G\times\mathcal H$.
Here, $\mathbf 1$ denotes the form which corresponds to the identity operator in $\mathcal G\otimes\mathcal H$.

\begin{remark}\label{sre} Just as in Remark~\ref{rtdtrd}, if there exists an operator $Q\in\mathscr L(\mathcal G\otimes\mathcal H)$ such that
$$(Q G_1\otimes F_1,G_2\otimes F_2)_{\mathcal G\otimes\mathcal H}=
 G(Z)(G_1,F_1,G_2,F_2),\quad G_1,G_2\in\mathcal G,\ F_1,F_2\in\mathcal H, $$then we can identify $ G(Z)$ with the operator $Q$.
\end{remark}

The following proposition shows that the generating function uniquely identifies the sequence $(M^{(n)})_{n=1}^\infty$.

\begin{proposition}\label{tyds}
Let $(M^{(n)})_{n=1}^\infty$ and $(\tilde M^{(n)})_{n=1}^\infty$  satisfy condition {\rm (B)}. Assume that
\begin{equation}\label{fytdfzwa}  G(Z)=\tilde G(Z),\quad Z\in\operatorname{Dom}( G)\cap \operatorname{Dom}(\tilde G).\end{equation}
 (Here, $\tilde G(Z)$ denotes the generating function of $(\tilde M^{(n)})_{n=1}^\infty$.)
 Then, for each $n\in\mathbb N$, $M^{(n)}=\tilde M^{(n)}$.
\end{proposition}

\begin{proof} Let $Z\in\mathcal Z(T)$. Fix $\varepsilon>0$ such that, for each $a\in(-\varepsilon,\varepsilon)$, $aZ\in \operatorname{Dom}( G)\cap \operatorname{Dom}(\tilde G)$. Then, by \eqref{fytdfzwa}, for each $G_1,G_2\in\mathcal G$, $F_1,F_2\in\mathcal H$ and each $a\in(-\varepsilon,\varepsilon)$,
$$\sum_{n=1}^\infty a^n\int_{T^n}Z^{\circledast n}\otimes dM^{(n)}(G_1,F_1,G_2,F_2)=\sum_{n=1}^\infty a^n\int_{T^n}Z^{\circledast n}\otimes d\tilde M^{(n)}(G_1,F_1,G_2,F_2).$$
Hence, for each $n\in\mathbb N$,
$$\int_{T^n}Z^{\circledast n}\otimes dM^{(n)}=\int_{T^n}Z^{\circledast n}\otimes d\tilde M^{(n)},\quad Z\in\mathcal Z(T).$$

Now, take as Hilbert space $\mathcal G$  the full Fock space over $\ell_2$: $\mathcal G=\mathcal F(\ell_2)$. Fix $n\in\mathbb N$ and choose any mutually orthogonal vectors $e_1,\dots,e_n$ in $\ell_2$ with norm 1.
Fix arbitrary $\Delta_1,\dots,\Delta_n\in\mathcal B_0(T)$ and define $Z\in\mathcal Z(T)$ by
$$Z(t):=\sum_{i=1}^n a^+(e_i)\chi_{\Delta_i}(t),$$
$a^+(e_i)$ being the creation operator at $e_i$. Set $G_1:=\Omega$ --- the vacuum, and $G_2:=e_1\otimes e_2\otimes\dots\otimes e_n$. Then, for any $F_1,F_2,\in\mathcal H$,
\begin{align}
&\int_{T^n} Z^{\circledast n}\otimes dM^{(n)}(G_1,F_1,G_2,F_2)\notag\\
&\quad=\sum_{i_1,\,i_2,\,\dots,i_n=1,\dots, n}\left(e_{i_1}\otimes e_{i_2}\otimes\dots\otimes e_{i_n},e_1\otimes e_2\otimes\dots\otimes e_n\right)_{\mathcal F(\ell_2)}\notag\\
&\qquad\times
M^{(n)}_{F_1,F_2}(\Delta_{i_1}\times \Delta_{i_2}\times\dots\times \Delta_{i_n}) \notag\\
&\quad=M^{(n)}_{F_1,F_2}(\Delta_{1}\times \Delta_{2}\times\dots\times \Delta_{n}) .\notag
\end{align}
  Therefore,
 $$M^{(n)}_{F_1,F_2}(\Delta_{1}\times \Delta_{2}\times\dots\times \Delta_{n}) =
 \tilde M^{(n)}_{F_1,F_2}(\Delta_{1}\times \Delta_{2}\times\dots\times \Delta_{n}) .$$
 Hence, by (B), for any $\Delta\in\mathcal B_0(T^n)$,
 $$ M^{(n)}_{F_1,F_2}(\Delta)=  \tilde M^{(n)}_{F_1,F_2}(\Delta) ,$$
 which implies the proposition.
\end{proof}

\section{Generating function for a free Meixner process}\label{fyfr}

We start with a brief recalling of the construction of a free Meixner process from \cite{bl}. Let $T$ be as in Section~\ref{itiutr68}, and we denote $\mathcal D:=C_0(T)$. Let $\sigma$ be a Radon non-atomic measure on $(T,\mathcal B(T))$  which satisfies $\sigma(O)>0$ for each open, non-empty set $O$ in $T$.
Fix any functions $\lambda,\eta\in C(T)$, which play the role of parameters of the free Meixner process. Consider the extended Fock space
$$\mathbb F=\mathbb R\oplus \bigoplus_{n=1}^\infty L^2(T^n,\gamma_n).$$
Here, each measure $\gamma_n$ on $(T^n,\mathcal B(T^n))$ is defined as in \cite{bl} through the function
$\eta$. In particular, $\gamma_n=\sigma^{\otimes n}$ if and only if  $\eta\equiv 0$.

The free Meixner process is defined as the family $(X(f))_{f\in\mathcal D}$ of bounded linear operators in $\mathbb F$ given  by
$$ X(f)=X^+(f)+X^0(f)+X^-(f),$$
where the creation operator $X^+(f)$, the neutral operator $X^0(f)$ and the (extended) annihilation operator $X^-(f)$ are defined by formulas (4.1)--(4.3) in \cite{bl}. We also have a representation of each $X(f)$ as
$$ X(f)=\int_T \sigma(dt) f(t)\omega(t)=\langle \omega,f\rangle,$$
where
\begin{equation}\label{cfdjb} \omega(t)=\di_t^\dag+\lambda(t)\di_t^\dag\di_t+\di_t+\eta(t)\di_t^\dag\di_t\di_t\end{equation}
with $\di_t^\dag$ and $\di_t$ being the creation and annihilation operator at point $t$, respectively (see
\cite[Corollary~4.2]{bl}).

The corresponding system of orthogonal polynomials is denoted in this paper as $$\langle P^{(n)}(\omega),f^{(n)}\rangle,\quad  f^{(n)}\in \mathcal D^{(n)}:=C_0(T^n),\ n\in\mathbb N_0. $$
These are the bounded linear operators in $\mathbb F$ which are recursively defined through
\begin{gather*}
P^{(0)}(\omega)=\mathbf 1,\quad  P^{(1)}(\omega)(t)=\omega(t),\\
P^{(n)}(\omega)(t_1,\dots,t_n)=\omega(t_1)P^{(n-1)}(\omega)(t_2,\dots,t_n)-\delta(t_1,t_2)
\lambda(t_1)P^{(n-1)}(\omega)(t_2,\dots,t_n)\\
\text{}-\delta(t_1,t_2)P^{(n-2)}(\omega)(t_3,\dots,t_n)
-[n-2]_0\delta(t_1,t_2,t_3)\eta(t_1)P^{(n-2)}(\omega)(t_3,\dots,t_n),\quad n\ge2,
\end{gather*}
where $\delta(t_1,t_2)$ and $\delta(t_1,t_2,t_3)$ are the `delta-functions' defined as in \cite[Section~2]{bl}.
In particular, for any $f_1,\dots,f_n\in\mathcal D$, $n\ge2$,
\begin{gather}
\la P^{(n)}(\omega),f_1\otimes\dots\otimes f_n\ra= \la\omega,f_1\ra\, \la P^{(n-1)}(\omega),f_2\otimes\dots\otimes f_n\ra\notag \\ \text{}-\la P^{(n-1)}(\omega),(\lambda f_1 f_2)\otimes f_3\otimes\dots\otimes f_n\ra-\int_T f_1(t)f_2(t)\,\sigma(dt)\,\la P^{(n-2)}(\omega),f_3\otimes\dots\otimes f_n\ra\notag\\
\text{}-[n-2]_0\la P^{(n-2)}(\omega),(\eta f_1f_2f_3)\otimes  f_4\otimes\dots\otimes f_n\ra. \label{fdser}
\end{gather}
Recall also that we may extend the definition of $X(f)$ and of $\la P^{(n)}(\omega),f^{(n)}\ra$ to the case where $f\in B_0(T)$ and $f^{(n)}\in B_0(T^n)$, respectively.

Our aim now is to derive the generating function for  these orthogonal polynomials.
So, let us fix a Hilbert space $\mathcal G$. From now on, for simplicity of notation, we will sometimes identify operators $\mathscr X\in\mathscr L(\mathcal G)$ and $\mathscr Y\in\mathscr L (\mathbb F)$ with the operators $\mathscr X\otimes\mathbf 1$ and $\mathbf 1\otimes\mathscr Y$ in $\mathscr L(\mathcal G\otimes\mathbb F)$.

For each $f\in\mathcal D$, we clearly have $\la \omega,f\ra=\int_T f\,dM$, where for each $\Delta\in\mathcal B_0(T)$, $M
(\Delta):=X(\chi_\Delta)$.
Note that $M
$ satisfies conditions
(A1)--(A3) with $$U(\Delta)=X^+(\chi_\Delta)+X^0(\chi_\Delta),\quad V(\Delta)=X^-(\chi_\Delta).$$
Therefore, by subsec.~\ref{tyrdert6}, we define, for each $Z\in\mathcal Z(T)$,
$$\la\omega,Z\ra:=\int_T Z\otimes dM
\in\mathscr L(\mathcal G\otimes\mathbb F).$$
It easily follows from \eqref{hcgdxgtrjdx} and the definition of the space $\mathbb F$ that
\begin{equation}\label{xsrzar}
\|\la\omega,Z\ra\|_{\mathscr L(\mathcal G\otimes\mathbb F)}\le \|Z\|_\infty C_3(\operatorname{supp}Z),\quad Z\in\mathcal Z,
\end{equation}
where
\begin{equation}\label{fxfx}
C_3(A):=2\sqrt{\sigma(A)}+2\sup_{t\in A}\eta(t)+\sup_{t\in A}|\lambda(t)|,\quad A\in\mathcal B_0(T).
\end{equation}

For any $n\in\mathbb N$ and any $Z_1,\dots,Z_n\in\mathcal Z(T)$, we  recurrently define an operator $\la P^{(n)}(\omega),Z_1\circledast\dots\circledast Z_n\ra$ from $\mathscr L(\mathcal G\otimes\mathbb F)$ as follows. By analogy with \eqref{fdser}, we set 
$\langle P^{(0)}(\omega),Z^{\circledast0}\rangle:=\mathbf 1$, 
$\langle P^{(1)}(\omega),Z\rangle:=\langle\omega,Z\rangle$ and for $n\ge2$
\begin{gather}
\la P^{(n)}(\omega),Z_1\circledast\dots\circledast Z_n\ra= \la\omega,Z_1\ra\, \la P^{(n-1)}(\omega),Z_2\circledast\dots\circledast Z_n\ra\notag \\ \text{}-\la P^{(n-1)}(\omega),(\lambda Z_1 Z_2)\circledast Z_3\otimes\dots\otimes Z_n\ra\notag\\
\text{}-
\int_T Z_1(t)Z_2(t)\,\sigma(dt)\,\la P^{(n-2)}(\omega),Z_3\circledast\dots\circledast Z_n\ra\notag\\
\text{}-[n-2]_0\la P^{(n-2)}(\omega),(\eta Z_1Z_2Z_3)\circledast  Z_4\circledast\dots\circledast Z_n\ra. \label{ufdycfrtd}
\end{gather}
Note that, for any $Z_1,Z_2\in\mathcal Z(T)$, the point-wise (non-comutative) product $Z_1Z_2$  belongs to $\mathcal Z(T)$, and for each $Z\in\mathcal Z(T)$, $\lambda Z$ and $\eta Z$ also belong to $\mathcal Z(T)$. In formula \eqref{ufdycfrtd} and below, for each $Z\in\mathcal Z(T)$, the integral
$\int_T Z(t)\,\sigma(dt)$ is understood in Bochner's sense, see e.g.\ \cite{BUS,Yosida}.

 It then easily follows by induction from \eqref{xsrzar}--\eqref{ufdycfrtd} and a standard estimate of the norm of a Bochner integral that, for any $A\in\mathcal B_0(T)$,  $n\in\mathbb N$, and any $Z_1,\dots,Z_n\in\mathcal Z(T)$ with support in $A$:
\begin{equation}\label{ydtr}
\|\la P^{(n)}(\omega),Z_1\circledast\dots\circledast Z_n\ra\|_{\mathscr L(\mathcal G\otimes\mathbb F)}\le C_4(A)^n\|Z_1\|_{\infty}\dotsm\|Z_n\|_\infty,
\end{equation}
where
\begin{align}C_4(A):&= C_3(A)+\sigma(A)+ \sup_{t\in A}|\lambda(t)|+\sup_{t\in A}\eta(t)\notag\\
&=2\sqrt{\sigma(A)}+\sigma(A)+2\sup_{t\in A}\eta(t)+3\sup_{t\in A}|\lambda(t)|
.\label{gzass} \end{align}
Hence, for each $Z\in\mathcal Z(T)$ such that $\|Z\|_\infty C_4(\operatorname{supp}Z)<1$, the sum
\begin{equation}\label{aw} G(Z)=\mathbf 1+\sum_{n=1}^\infty \la P^{(n)}(\omega),Z^{\circledast n}\ra\end{equation}
defines an operator from $\mathscr L(\mathcal G\otimes\mathbb F)$.

Next, we set, for each $n\in\mathbb N$ and $\Delta\in\mathcal B_0(T^n)$,
$$ M^{(n)}(\Delta):=\la P^{(n)}(\omega),\chi_\Delta\ra.$$
Analogously to  \eqref{ydtr}, we conclude that the sequence $(M^{(n)})_{n=1}^\infty$ satisfies condition (B), and so the function $ G$ defined by \eqref{aw} is the generating function of the operator-valued measures $(M^{(n)})_{n=1}^\infty$ in the sense of subsec.~\ref{fstrs}.
Hence, by Proposition~\ref{tyds}, the generating function $ G$ uniquely identifies $(M^{(n)})_{n=1}^\infty$, and hence also  polynomials $\la P^{(n)}(\omega),f^{(n)}\ra$, $f^{(n)}\in\mathcal D^{(n)}$.

To stress the dependence of the generating function $ G(Z)$ on the free generalized stochastic process $\omega$, we will  write $ G(Z,\omega)$.

\begin{theorem}\label{xdf}
Fix any  $A\in\mathcal B_0(T)$. Then there exits a constant $C_5(A)>0$ such that, for any $Z\in\mathcal Z(T)$ satisfying $\operatorname{supp} Z\subset A$ and $\|Z\|_\infty< C_5(A)$, formula \eqref{ftydr} holds. Furthermore, we have
\begin{equation}\label{ftsre}
 G(Z,\omega)=\left(\mathbf 1-f(Z)\big\la \omega(\cdot),\Psi_{\lambda(\cdot),\eta(\cdot)}(Z(\cdot))\big\ra\right)^{-1}f(Z),
\end{equation}
where
\begin{equation}\label{drt}f(Z):=\left(\mathbf 1+\int_T\frac{Z(t)^2}{1+\lambda(t)Z(t)+\eta(t)Z(t)^2}\,\sigma(dt)\right)^{-1}. \end{equation}
\end{theorem}

\begin{remark}\label{tser}
The right hand side of formula \eqref{ftydr} should be understood in the following sense: for any real-valued function $f(x)=\sum_{n=0}^\infty a_nx^n$ which is real-analytic on $(-r,r)$, we write, for a bounded linear operator $B$ whose norm is less than $r$:
$f(B):=\sum_{n=0}^\infty a_n B^n$.
Under our assumption on $Z\in\mathcal Z(T)$, we then have $\frac{Z^l}{\mathbf  1+\lambda Z+\eta Z^2}\in\mathcal Z(T)$, $l=1,2$.
\end{remark}

\begin{proof} We divide the proof into several steps.

Step 1. First, for a fixed $A\in\mathcal B_0(T)$, let us explicitly specify a possible choice of a  constant $C_5(A)$ in the theorem. For each $t\in T$, define $\alpha(t),\beta(t)\in\mathbb C$ so that
$$\alpha(t)+\beta(t)=\lambda(t),\quad \alpha(t)\beta(t)=\eta(t).$$
Hence, for each $x\in\mathbb R$,
$$1+\lambda(t)x+\eta(t)x^2=(1-\alpha(t)x)(1-\beta(t)x).$$
The right hand side of formula \eqref{ftydr} now reads as
\begin{multline}\label{gse}\left(\mathbf 1-\langle \omega,Z(1-\alpha Z)^{-1}(1-\beta Z)^{-1}\rangle\right.\\ \left.\vphantom{\mathbf 1-\langle \omega,Z(1-\alpha Z)^{-1}(1-\beta Z)^{-1}\rangle}+\int_T Z(t)^2(1-\alpha(t)Z(t))^{-1}(1-\beta(t)Z(t))^{-1}\,\sigma(dt)\right)^{-1}\end{multline}
(we consider the above operator in the complexification of the real Hilbert space $\mathcal G\otimes \mathbb F$, for which we keep the same notation).
Set
$$\alpha_A:=\sup_{t\in A}|\alpha(t)|,\quad \beta_A:=\sup_{t\in A}|\beta(t)|.$$
Choose $C_6(A)>0$ so that
$$ \left(\sum_{k=0}^\infty \alpha_A^k C_6(A)^k\right)  \left(\sum_{l=0}^\infty \beta_A^l C_6(A)^l\right)
C_6(A)\big(C_3(A)+C_6(A)\sigma(A)\big)<1.
$$
Then, by virtue of \eqref{xsrzar}, we have that, for each $Z\in\mathcal Z(T)$  such that $\operatorname{supp}Z\subset A$ and $\|Z\|_\infty\le C_6(A)$, formula \eqref{gse} defines a bounded linear operator in $\mathscr L(\mathcal G\otimes\mathbb F)$.
Recall \eqref{ydtr}--\eqref{aw}.  Then setting
\begin{equation}\label{fdst} C_5(A):=\min\{C_4(A)^{-1},C_6(A)\}, \end{equation}
we see that, for each $Z\in\mathcal Z(T)$  such that $\operatorname{supp}Z\subset A$ and $\|Z\|_\infty< C_5(A)$, the left and right hand sides of formula \eqref{ftydr} identify bounded linear operators in $ \mathcal G\otimes\mathbb F$.

Let us denote the operators on the left and right hand sides of formula \eqref{ftydr} by $L(Z)$ and $R(Z)$, respectively.
Fix any $\Xi,\Upsilon\in\mathcal G\otimes\mathbb F$.  It follows that, for any $Z\in\mathcal Z(T)$ such that
$\operatorname{supp}Z\subset A$, the functions
$$ f^{(L)}(z):=(L(zZ)\Xi,\Upsilon)_{\mathcal G\otimes\mathbb F},\quad f^{(R)}(z):=(R(zZ)\Xi,\Upsilon)_{\mathcal G\otimes\mathbb F}$$
are analytic on $\left\{z\in\mathbb C: |z|<C_5(A)\|Z\|^{-1}_\infty\right\}$.

Step 2. Fix any $A\in\mathcal B_0(T)$. Choose any set partition $\mathscr P=\{\Delta_1,\dots,\Delta_J\}$ of $A$, i.e.,  $$A=\bigcup_{j=1}^J\Delta_j,\quad \Delta_j\in\mathcal B_0(T),\ j=1,\dots,J,\ J\in\mathbb N$$
and the sets $\Delta_j$ are mutually disjoint.
Set $$\lambda_j:=\inf_{t\in \Delta_j}\lambda(t),\quad \eta_j:=\inf_{t\in \Delta_j}\eta(t), \quad j=1,\dots,J,$$
and define a function
 $$\lambda_{\mathscr P}(t):=\begin{cases}
\lambda_j,&\text{if }t\in \Delta_j,\ j=1,\dots,J,\\
0,&\text{if }t\in A^c,
 \end{cases}$$
 and analogously a function $\eta_{\mathscr P}(t)$.
 Now, we define a generalized operator-valued process $\omega_{\mathscr P}(t)$ and corresponding non-commutative polynomials $\la P^{(n)}(\omega_{\mathscr P}),f^{(n)}\ra$, $f^{(n)}\in B_0(T^n)$, in the same way as $\omega(t)$ and $\la P^{(n)}(\omega),f^{(n)}\ra$ were defined, but by using the functions $\lambda_{\mathscr P}$ and $\eta_{\mathscr P}$ instead of $\lambda$ and $\eta$, respectively. We stress that these are also defined in the extended Fock space $\mathbb F$ constructed through the function $\eta$. Hence, generally speaking, the operators $\la P^{(n)}(\omega_{\mathscr P}),f^{(n)}\ra$ are not self-adjoint in $\mathbb F$.
  This, however, does not lead to any  problem when we define a generating function $ G_{\mathscr P}(Z)$ of these polynomials. In particular, the corresponding operator-valued measure
$$ M_{\mathscr P}(\Delta):=\la \omega_{\mathscr P},\chi_\Delta\ra,\quad \Delta\in \mathcal B_0(T),$$
satisfies conditions (A1)--(A3) with $U(\Delta)=X^+(\chi_\Delta)+X_{\mathscr P}^0(\chi_\Delta)$ and $V(\Delta)=X_{\mathscr P}^-(\Delta)$, where
\begin{gather*}
X^+(\chi_\Delta):=\int_\Delta\di_t^\dag\,\sigma(dt),\quad X^0_{\mathscr P}(\chi_\Delta):=\int_\Delta\lambda_{\mathscr P}(t)\di_t^\dag\di_t\,\sigma(dt),\\
X^-_{\mathscr P}(\chi_\Delta):=\int_\Delta(\di_t+\eta_{\mathscr P}(t)\di_t^\dag\di_t\di_t)\,\sigma(dt),
\end{gather*}
compare with \eqref{cfdjb}.
(We leave the evaluation of the adjoint operator of $X^-_{\mathscr P}(\chi_\Delta)$ in $\mathbb F$ to the interested reader.)
Furthermore, analogously to \eqref{ydtr}, we get, for any $A\in\mathcal B_0(T)$,  $n\in\mathbb N$ and any $Z_1,\dots,Z_n\in\mathcal Z(T)$ with support in $A$,
\begin{equation}\label{hfch}
\|\la P^{(n)}(\omega_{\mathscr P}),Z_1\circledast\dots\circledast Z_n\ra\|_{\mathscr L(\mathcal G\otimes\mathbb F)}\le C_4(A)^n\|Z_1\|_{\infty}\dotsm\|Z_n\|_\infty,
\end{equation}
with  the same constant $C_4(A)$ given by \eqref{gzass}. (We, in particular, used that $\eta_{\mathscr P}(t)\le \eta(t)$ for all $t\in A$.)

 Step 3. By definition, for each $j=1,\dots,J$, the polynomials $\big(\la P^{(n)}(\omega_{\mathscr P}),\chi_{\Delta_j}^{\otimes n}\ra\big)_{n=1}^\infty$ satisfy the recursion relation
  \begin{multline*}
 \la P^{(n)}(\omega_{\mathscr P}),\chi_{\Delta_j}^{\otimes n}\ra=\big(\la \omega_{\mathscr P},\chi_{\Delta_j}\ra-\lambda_j\big)\la P^{(n-1)}(\omega_{\mathscr P}),\chi_{\Delta_j}^{\otimes (n-1)}\ra
\\
\text{} -(\sigma(\Delta_j)+[n-2]_0\,\eta_j)\la P^{(n-2)}(\omega_{\mathscr P}),\chi_{\Delta_j}^{\otimes (n-2)}\ra,\quad n\ge2.
  \end{multline*}
 Therefore,
 \begin{equation}\label{ydtrcg}
 \la P^{(n)}(\omega_{\mathscr P}),\chi_{\Delta_j}^{\otimes n}\ra=P^{(n)}_{\lambda_j,\eta_j,\sigma(\Delta_j)}(\la \omega_{\mathscr P},\chi_{\Delta_j}\ra),
 \end{equation}
 where  $(P^{(n)}_{\lambda_j,\eta_j,\sigma(\Delta_j)})_{n=0}^\infty$ is a system of polynomials on $\mathbb R$ recursively defined by \begin{gather}P^{(0)}_{\lambda_j,\eta_j,\sigma(\Delta_j)}(u)=1,\quad P^{(1)}_{\lambda_j,\eta_j,\sigma(\Delta_j)}(u)=u,\notag \\
 P^{(n)}_{\lambda_j,\eta_j,\sigma(\Delta_j)}(u)=(u-\lambda_j)P^{(n-1)}_{\lambda_j,\eta_j,\sigma(\Delta_j)}(u)
 -(\sigma(\Delta_j)+[n-2]_0\,\eta_j)
 P^{(n-2)}_{\lambda_j,\eta_j,\sigma(\Delta_j)}(u),\quad n\ge2. \label{ftdrd}
 \end{gather}
 By \cite{a1}, the generating function of $(P^{(n)}_{\lambda_j,\eta_j,\sigma(\Delta_j)})_{n=0}^\infty$ is given by
 \begin{equation}\label{ctsrs}\sum_{n=0}^\infty z^{n}P^{(n)}_{\lambda_j,\eta_j,\sigma(\Delta_j)}(u)=\left(
1-u\,\frac{z}{1+\lambda_j z+\eta_j z^2}+\frac{\sigma(\Delta_j)z^2}{1+\lambda_j z+\eta_j z^2}
\right)^{-1}. \end{equation}
More precisely, for each $r>0$, there exists  $\varepsilon_{r,A}>0$ such that formula \eqref{ctsrs} holds for each $u\in\mathbb R$ with $|u|\le r$ and for each $z\in\mathbb C$ such that $|z|<\varepsilon_{r,A}$.

Let $Z_j\in\mathscr L(\mathcal G)$ be such that $\|Z_j\|_{\mathscr L(\mathcal G)}< C_5(A)$, where $C_5(A)$ is given by \eqref{fdst}. Then, by \eqref{ydtrcg} and \eqref{ctsrs}, we  get
\begin{multline}
\mathbf 1+\sum_{n=1}^\infty Z_j^n\la P^{(n)}(\omega_{\mathscr P}),\chi_{\Delta_j}^{\otimes n}\ra
\\=\left(
\mathbf 1-\la\omega_{\mathscr P},\chi_{\Delta_j}\ra\,\frac{Z_j}{\mathbf 1+\lambda_j Z_j+\eta_j Z_j^2}+\sigma(\Delta_j)\,\frac{Z_j^2}{1+\lambda_j Z_j+\eta_j Z_j^2}
\right)^{-1},\quad j=1,\dots,J.\label{fytd}
\end{multline}
Denote
$$ U_j:=\la\omega_{\mathscr P},\chi_{\Delta_j}\ra\,\frac{Z_j}{\mathbf 1+\lambda_j Z_j+\eta_j Z_j^2}+\sigma(\Delta_j)\,\frac{Z_j^2}{1+\lambda_j Z_j+\eta_j Z_j^2}\, ,\quad j=1,\dots,J.$$
Then \eqref{fytd} is equivalent to
\begin{equation}\label{gfarea}
\sum_{n=1}^\infty \langle P^{(n)}(\omega_{\mathscr P}), (Z_j\chi_{\Delta_j})^{\circledast n}\rangle=\sum_{n=1}^\infty U_j^n,\quad j=1,\dots,J.
\end{equation}

Step 4. We claim that,  for any $n\in\mathbb N$ and any $j_1,j_2,\dots,j_n\in\{1,2,\dots,J\}$ such that $j_1\ne j_2$, $j_2\ne j_3$,\dots, $j_{n-1}\ne j_n$, and any $k_1,k_2,\dots k_n\in\mathbb N$, we have
\begin{multline}\label{cxts}
\langle P^{(k_1+k_2+\dots+k_n)}(\omega_{\mathscr P}),\chi_{\Delta_{j_1}}^{\otimes k_1}\otimes \chi_{\Delta_{j_2}}^{\otimes k_2}\otimes\dots\otimes\chi_{\Delta_{j_n}}^{\otimes k_n}\ra\\
=\la P^{(k_1)}(\omega_{\mathscr P}),\chi_{\Delta_{j_1}}^{\otimes k_1}\ra \la P^{(k_2)}(\omega_{\mathscr P}),\chi_{\Delta_{j_2}}^{\otimes k_2}\ra\dotsm \la P^{(k_n)}(\omega_{\mathscr P}),\chi_{\Delta_{j_n}}^{\otimes k_n}\ra.
\end{multline}
Indeed, first we can prove by induction in $k_1\in\mathbb N$ that, for any fixed $k_2\in\mathbb N$,
and any $j_1,j_2\in\{1,2,\dots,J\}$, $j_1\ne j_2$,
$$ \langle P^{(k_1+k_2)}(\omega_{\mathscr P}),\chi_{\Delta_{j_1}}^{\otimes k_1}\otimes \chi_{\Delta_{j_2}}^{\otimes k_2}\ra=\la P^{(k_1)}(\omega_{\mathscr P}),\chi_{\Delta_{j_1}}^{\otimes k_1}\ra \la P^{(k_2)}(\omega_{\mathscr P}),\chi_{\Delta_{j_2}}^{\otimes k_2}\ra. $$
Then, we prove \eqref{cxts} by induction in $n\in\mathbb N$.

Step 5. Now, fix any $Z_1,\dots,Z_J\in\mathscr L(\mathcal G)$ such that
\begin{equation}\label{xsre} \max_{j=1,\dots,J}\|Z_j\|_{\mathscr L(\mathcal G)}<\frac{C_5(A)}J\,.\end{equation}
Then, as easily seen,
$$\max_{j=1,\dots,J}\|U_j\|_{\mathscr L(\mathcal G\otimes\mathbb F)}<\frac1J\,.$$
By \eqref{gfarea} and \eqref{cxts}, we have:
\begin{align}
&\sum_{n=1}^\infty (U_1+U_2+\dots+U_J)^n\notag\\
&\qquad =\sum_{n=1}^\infty\sum_{\substack{
j_1,j_2,\dots,j_n\in\{1,2,\dots,J\}
\\
j_1\ne j_2,\ j_2\ne j_3,\dots,\ j_{n-1}\ne j_n}}
\left(\sum_{k_1=1}^\infty U_{j_1}^{k_1}\right)\left(\sum_{k_2=1}^\infty U_{j_2}^{k_2}\right)\dotsm
\left(\sum_{k_n=1}^\infty U_{j_n}^{k_n}\right)\notag\\
&\qquad =\sum_{n=1}^\infty\sum_{\substack{
j_1,j_2,\dots,j_n\in\{1,2,\dots,J\}
\\
j_1\ne j_2,\ j_2\ne j_3,\dots,\ j_{n-1}\ne j_n}}
\left(\sum_{k_1=1}^\infty \la P^{(k_1)}(\omega_{\mathscr P}),(Z_{j_1}\chi_{\Delta_{j_1}})^{\circledast k_1}\ra\right)\notag\\
&\qquad\quad\times\left(\sum_{k_2=1}^\infty \la P^{(k_2)}(\omega_{\mathscr P}),(Z_{j_2}\chi_{\Delta_{j_2}})^{\circledast k_2}\ra\right)\dotsm \left(\sum_{k_n=1}^\infty \la P^{(k_n)}(\omega_{\mathscr P}),(Z_{j_n}\chi_{\Delta_{j_n}})^{\circledast k_n}\ra\right)\notag\\
&\qquad =\sum_{n=1}^\infty\sum_{\substack{
j_1,j_2,\dots,j_n\in\{1,2,\dots,J\}
\\
j_1\ne j_2,\ j_2\ne j_3,\dots,\ j_{n-1}\ne j_n}}
\sum_{k_1=1}^\infty \sum_{k_2=1}^\infty\dotsm \sum_{k_n=1}^\infty\notag\\
&\qquad\quad\times
\la P^{(k_1+k_2+\dots+k_n)}(\omega_{\mathscr P}), (Z_{j_1}\chi_{\Delta_{j_1}})^{\circledast k_1}
\circledast (Z_{j_2}\chi_{\Delta_{j_2}})^{\circledast k_2}\circledast\dots\circledast (Z_{j_n}\chi_{\Delta_{j_n}})^{\circledast k_n}\ra
\notag\\
&\qquad=\sum_{n=1}^\infty
\la P^{(n)}(\omega_{\mathscr P}),(Z_{1}\chi_{\Delta_{1}}+Z_{2}\chi_{\Delta_{2}}+\dots+Z_{n}\chi_{\Delta_{n}})^{\circledast n}\ra.\notag
\end{align}
Setting
\begin{equation}\label{gxsre} Z(t)=Z_1\chi_{\Delta_1}(t)+Z_2\chi_{\Delta_1}(t)+\dots+Z_J\chi_{\Delta_J}(t),\quad t\in T,\end{equation}
we thus get
\begin{multline}\label{vcdj}\mathbf 1+\sum_{n=1}^\infty \la P^{(n)}(\omega_{\mathscr P}),Z^{\circledast n}\ra\\=
\left(
\left\la \omega_{\mathscr P},\frac Z{1+\lambda_{\mathscr P} Z+\eta_{\mathscr P}Z^2}\right\ra+\int_T
\frac {Z(t)^2}{1+\lambda_{\mathscr P}(t) Z(t)+\eta_{\mathscr P}(t)Z(t)^2}\,\sigma(dt)
\right)^{-1},
\end{multline}
provided \eqref{xsre} holds.

Step 6. Denote the left and right hand sides of formula \eqref{vcdj} by $L_{\mathscr P}(Z)$ and
$R_{\mathscr P}(Z)$, respectively. Analogously to Step~1, we see that  $L_{\mathscr P}(Z)$ and
$R_{\mathscr P}(Z)$ are in $\mathscr L(\mathcal G\otimes\mathbb F)$ for any $Z\in\mathcal Z(T)$ such that $\operatorname{supp}Z\subset A$ and $\|Z\|_\infty< C_5(A)$, and furthermore, for
fixed  $\Xi,\Upsilon\in\mathcal G\otimes\mathbb F$ and any $Z\in\mathcal Z$ such that
$\operatorname{supp}Z \subset A$, the functions
$$ f^{(L)}_{\mathscr P}(z):=(L_{\mathscr P}(zZ)\Xi,\Upsilon)_{\mathcal G\otimes\mathbb F},\quad
f^{(R)}_{\mathscr P}(z):=(R_{\mathscr P}(zZ)\Xi,\Upsilon)_{\mathcal G\otimes\mathbb F}
$$
are  analytic on $\left\{z\in\mathbb C: |z|<C_5(A)\|Z\|^{-1}_\infty\right\}$.
By \eqref{vcdj}, we have
\begin{equation}\label{ydserw}f^{(L)}_{\mathscr P}(z) =f^{(R)}_{\mathscr P}(z) \end{equation}
for all $z\in\mathbb C$ satisfying $|z|<C_5(A)J^{-1} \|Z\|^{-1}_\infty$. Therefore, equality \eqref{ydserw}
holds for all $z\in \mathbb C$ satisfying $|z|<C_5(A)\|Z\|^{-1}_\infty$. Hence, \eqref{vcdj} holds for any $Z\in\mathcal Z(T)$ as in \eqref{gxsre}, provided that
$$ \max_{j=1,\dots,J}\|Z_j\|_{\mathscr L(\mathcal G)}< C_5(A). $$

Step 7. Now, fix any $Z\in\mathcal Z(T)$ as above and consider any partition $\mathscr P'$ of $A$ which is finer than $\mathscr P$, i.e., any element $\Delta\in \mathscr P'$ is a subset of some $\Delta_j\in\mathscr P$. Evidently, formula \eqref{vcdj} remains true if we replace $\omega_{\mathscr P}$ with $\omega_{\mathscr P'}$ in it. For each $n\in\mathbb N$, denote by $\mathscr P_n$ a partition of $A$ which is finer than $\mathscr P$ and such that, for each $\Delta\in \mathscr P_n$,
$$\left(\sup_{s,t\in\Delta}|\lambda(t)-\lambda(s)|\right)\vee\left(\sup_{s,t\in\Delta}|\eta(t)-\eta(s)|\right)\le\frac 1n\, .$$
(Clearly, such  $\mathscr P_n$ exists.) By the dominated convergence theorem,
$$L_{\mathscr P_n}(Z)\to L(Z), \quad R_{\mathscr P_n}(Z)\to R(Z)\quad\text{as }n\to\infty$$
in $\mathscr L(\mathcal G\otimes\mathbb F)$. Hence, formula \eqref{ftydr} holds for any simple mapping $Z\in\mathcal Z(T)$ satisfying $\operatorname{supp} Z\subset A$ and $\|Z\|_\infty< C_5(A)$. Now, \ for a general $Z\in\mathcal Z(T)$, \eqref{ftydr} follows by approximation of $Z$ by simple mappings and the dominated convergence theorem.

Finally, formulas \eqref{ftsre}, \eqref{drt} follow directly from \eqref{ftydr}, since under our assumptions, the operator
$$\mathbf 1+\int_T\frac{Z(t)^2}{1+\lambda(t)Z(t)+\eta(t)Z(t)^2}\,\sigma(dt)$$
is invertible. \end{proof}

\begin{corollary}
Let $\Delta_1,\Delta_2,\dots,\Delta_n\in\mathcal B_0(T)$ ($n\ge2$) be such that $\Delta_i\cap\Delta_{i+1}=\varnothing$, $i=1,2,\dots,n-1$. 
Let $k_1,k_2,\dots,k_n\in\mathbb N$ and let, 
for each $i=1,2,\dots,n$, $g^{(k_i)}\in B_0(T^{k_i})$ vanish outside the set $\Delta_i^{k_i}$. Then
\begin{multline*}\la P^{(k_1+k_2+\dots+k_n)}(\omega),g^{(k_1)}\otimes g^{(k_2)}\otimes\dots\otimes g^{(k_n)}\ra\\
=
\la P^{(k_1)}(\omega),g^{(k_1)}\ra\la P^{(k_2)}(\omega),g^{(k_2)}\ra\dotsm \la P^{(k_n)}(\omega),g^{(k_n)}\ra.\end{multline*}
\end{corollary}

\begin{proof} The statement follows analogously to the proof of formula \eqref{cxts}.

\end{proof}

\section{A globality of the annihilation operators}\label{jfdrt}

Recall that the free Meixner process $(X(f))_{f\in\mathcal D}$ is a family of bounded liner operators acting in $\mathbb F$. In view of the unitary isomorphism between $\mathbb F$ and the non-commutative $L^2$-space $L^2(\tau)$ (see \cite{bl}), each $X(f)$ acts in $L^2(\tau)$ as the operator of left multiplication by $\la\omega,f\ra$. In view of the expansion \eqref{cfdjb} it is, in particular, desirable to better understand the action of the annihilation operators $\di_t$, $t\in T$. Each such operator is well defined as a linear operator acting on the set of continuous polynomials in $\omega$ (denoted by $\mathbf{CP}$ in \cite{bl}) through
$$\di_t \la P^{(n)}(\omega),f^{(n)}\ra =\la P^{(n-1)}(\omega),f^{(n)}(t,\cdot)\ra,\quad f^{(n)}\in\mathcal D^{(n)}.$$
(Recall that each continuous polynomial has a unique representation as a finite sum of  orthogonal polynomials $\la P^{(n)}(\omega),f^{(n)}\ra$ with $f^{(n)}\in\mathcal D^{(n)}$.)

In the classical case, in  one dimension, the annihilation operator $\di$, defined by $\di P^{(n)}(t)=nP^{(n)}(t)$, is an analytic function of the operator of differentiation, $D$. Indeed, it directly follows from \eqref{drtyd} that $\di=\Psi_{\lambda,\eta}^{-1}(D)$, cf.\  \cite{Meixner}. This result has its counterpart in the infinite-dimensional case: as follows from \eqref{uyfrutf}, the annihilation operator at a point $t$ satisfies \begin{equation}\label{fgsare}
\di_t=\Psi_{\lambda(t),\eta(t)}^{-1}(D_t),\end{equation}
 where $D_t$ is the operator of differentiation in the direction of the delta-function at $t$ (often called Hida--Malliavin  derivative), cf.\ \cite{Ly1,Rod}. In particular, if $\lambda(t)=\eta(t)=0$
(Gassian white noise at $t$),
\begin{equation}\label{dss} \di_t=D_t,\end{equation}
and more generally, if $\eta(t)=0$ (Poisson white noise if $\lambda(t)\ne0$)
\begin{equation}\label{dctdv} \di_t=\frac{1}{\lambda(t)}\left(e^{\lambda(t)D_t}-1\right)=\sum_{k=1}^\infty \frac{\lambda(t)^{k-1}}{k!}\, D_t^k.
\end{equation}

In the free case, in one dimension, one has
\begin{equation}\label{dst} G(z,t)=(1-t\Psi_{\lambda,\eta+k}(z))^{-1}f(z),\end{equation}
where 
\begin{align*}\Psi_{\lambda,\eta+k}(z)&=\frac{z}{1+\lambda z+(\eta+k)z^2},\\
 f(z)&=\frac{1+\lambda z+\eta z^2}{1+\lambda z+(\eta+k) z^2}\, .\end{align*}
 Therefore, the corresponding annihilation operator, defined by $\di P^{(n)}(t)=[n]_0P^{(n-1)}(t)$, satisfies $\di=\Psi_{\lambda,\eta+k}^{-1}(D)$, where $D$ now denotes the operator of free differentiation: $Dt^n=[n]_0t^{n-1}$, or equivalently $Df(t)=\frac{f(t)-f(0)}t$, cf. \cite{LR}.
In the infinite-dimensional case, we  define operators of free differentiation by setting, for each $t\in T$,
 \begin{equation}\label{gtse}D_t\la \omega^{\otimes n},f^{(n)}\ra=[n]_0\la \omega^{\otimes(n-1)},f^{(n)}(t,\cdot)\ra,\quad f^{(n)}\in \mathcal D^{(n)}.\end{equation}
However, a direct analogy with the classical case and the free one-dimensional case breaks down at this point: the operator $\di_t$ cannot be represented as a function of $D_t$. Indeed, the operator $f(Z)$ in \eqref{ftsre} is  `global': according to \eqref{drt}, $f(Z)$ depends  on the whole `trajectory' $(Z(s))_{s\in T}$. Still, in the free Gauss--Poisson case, we will now derive a free counterpart of formulas \eqref{fgsare}--\eqref{dctdv}.

So,  let $\eta\equiv0$. Let $NC_{\ge2}(1,2,\dots,n)$ denote the set of all non-crossing partitions of $\{1,2,\dots,n\}$ such that each element of a partition contains at least two points. Analogously to \cite{bl}, we define, for each $\zeta\in NC_{\ge2}(1,2,\dots,n)$,
$$W^-(\zeta)(t_1,\dots,t_n)=\prod_{l\ge 2}\prod_{\{i_1,i_2,\dots,i_l\}\in\zeta}\lambda^{l-2}(t_{i_1})\delta(t_{i_1},t_{i_2},\dots,t_{i_l}).$$
We define a linear operator $\mathbb G$ acting on $\mathbf{CP}$ (a `global' operator) by
\begin{equation}\label{ghfsaer}
\mathbb G:=\mathbf 1+\sum_{n=2}^\infty\sum_{\zeta\in NC_{\ge2}(1,2,\dots,n)}\int_{T^n}\sigma(dt_1)\sigma(dt_2)\dotsm \sigma(dt_n)W^-(\zeta)(t_1,\dots,t_n) D_{t_1}D_{t_2}\dotsm D_{t_n}.
\end{equation}In fact, by virtue of \eqref{gtse}, when the operator $\mathbb G$ acts on a polynomial from $\mathbf {CP}$, all but finitely many terms in the sum in \eqref{ghfsaer} vanish. For example,
\begin{align}
& \mathbb G\la \omega^{\otimes 4},f_1\otimes f_2\otimes f_3\otimes f_4\ra=\bigg(\mathbf 1+\int_{T^2}\sigma(dt_1)\sigma(dt_2)\delta(t_1,t_2)D_{t_1}D_{t_2}\notag\\
&\quad+\int_{T^3}\sigma(dt_1)
\sigma(dt_2)\sigma(dt_3)\lambda(t_1)\delta(t_1,t_2,t_3)D_{t_1}D_{t_2}D_{t_3}\notag\\
&\quad+\int_{T^4}\sigma(dt_1)
\sigma(dt_2)\sigma(dt_3)\sigma(dt_4)\big(\delta(t_1,t_2)\delta(t_3,t_4)+\delta(t_1,t_4)\delta(t_2,t_3)\notag\\
&\quad+\lambda(t_1)^2
\delta(t_1,t_2,t_3,t_4)\big)D_{t_1}D_{t_2}D_{t_3}D_{t_4}\bigg)\la\omega^{\otimes 4},f_1\otimes f_2\otimes f_3\otimes f_4\ra\notag\\
&=\la\omega^{\otimes 4},f_1\otimes f_2\otimes f_3\otimes f_4\ra+\la f_1f_2\ra_{\sigma}\la \omega^{\otimes 2},f_3\otimes f_4\ra\notag\\
&\quad+\la \lambda f_1f_2f_3\ra_\sigma
\la\omega,f_4\ra+\la f_1f_2\ra_\sigma \la f_3f_4\ra_\sigma
+\la f_1f_4\ra_\sigma \la f_2f_3\ra_\sigma+\la \lambda^2 f_1f_2f_3f_4\ra_\sigma,\label{ydr}
\end{align}
where $\la f\ra_\sigma:=\int_T f(t)\sigma(dt)$.

\begin{theorem}\label{sers}
Let $\eta\equiv0$. For each $t\in T$, the operator $\di_t$
acting on $\mathbf{CP}$ has the following representation:
\begin{equation}\label{hser}
\di_t=\Psi_{\lambda(t),0}^{-1}(D_t\mathbb G)=\frac{D_t\mathbb G}{\mathbf1-\lambda(t)D_t\mathbb G}=\sum_{k=1}^\infty \lambda(t)^{k-1}(D_t \mathbb G)^k.\end{equation}
In particular, if $\lambda(t)=0$,
$$\di_t=D_t \mathbb G.$$
\end{theorem}

\begin{remark}
When the operator $\sum_{k=1}^\infty \lambda(t)^{k-1}(D_t\mathbb G)^k$ acts on a polynomial from $\mathbf{CP}$, all but finitely many terms in the sum vanish.
\end{remark} 

\begin{remark} The reader is advised to compare formulas \eqref{dctdv} and \eqref{hser}. Recall that the free counterpart of $k!$ is $[k]_0!=1$.
\end{remark}

\begin{proof} First, we mention that, by  \eqref{kjfdry},
$ \Psi_{\lambda,0}(z)=z/(1+\lambda z)$,
and so
$ \Psi^{-1}_{\lambda,0}(z)=z/(1-\lambda z$.

Let now $n\in\mathbb N$. Recall that, in \cite{bl} (after Theorem~2.1), we introduced the class $G_n$ of $\pm1$-marked non-crossing partitions of $\{1,2,\dots,n\}$ such that each element of a partition with mark $-1$ has at least two elements and no element of a partition with mark $+1$ is `within' any other element of this partition. Analogously, for any set $\{i_1,i_2,\dots,i_k\}\subset\mathbb N$, we denote by $G(i_1,i_2,\dots,i_k)$ the corresponding class of marked partitions of $\{i_1,i_2,\dots,i_k\}$.
For any two disjoint sets $\{i_1,i_2,\dots,i_k\}$ and $\{j_1,j_2,\dots,j_l\}$
and any $\varkappa_1\in G(i_1,i_2,\dots,i_k)$, $\varkappa_2\in G(j_1,j_2,\dots,j_l)$, we
may consider $\varkappa_1\cup \varkappa_2$ as an element of
$ G(i_1,i_2,\dots,i_k,j_1,j_2,\dots,j_l )$, provided $\varkappa_1\cup \varkappa_2$ indeed satisfies the necessary conditions to be in this set.

We also denote by $G^-(i_1,i_2,\dots,i_k)$ the subset of $G(i_1,i_2,\dots,i_k)$ consisting of all marked partitions whose all elements have mark $-1$. (Note that each set from such a partition has at least two elements.)

By \cite[Theorem~2.3]{bl} and using the notation introduced in that paper, for any $f_1,\dots,f_n\in\mathcal D$, we get
\begin{align}
& \la \omega^{\otimes n}, f_1,\otimes\dots\otimes f_n\ra =\sum_{\varkappa\in G_n}\int_{T^n}\sigma(dt_1)
\dotsm \sigma(dt_n)\, {:} W(\varkappa) \omega(t_1)\dotsm \omega(t_n){:}\, f(t_1)\dotsm f(t_n)\notag\\
& \quad =\sum_{\substack{k=1,\dots,n\\ k\ne2}}\,\sum_{\varkappa_1\in G^-(1,2,\dots,k-1)}\sum_{m=1}^{n-k-1}
\sum_{\substack{ \{i_1,i_2,\dots,i_m\}\subset \{k,k+1,\dots,n\}\\ k=i_1<i_2<\dots<i_m}}\,
\sum_{\varkappa_2\in G^-(i_1+1,i_1+2,\dots,i_2-1)}\notag\\
&\quad\times \sum_{\varkappa_3\in G^-(i_2+1,i_2+2,\dots,i_3-1)}\dots\sum_{\varkappa_m\in G^-(i_{m-1}+1,i_{m-1}+2,\dots,i_m-1)}\sum_{\varkappa_{m+1}\in G(i_m+1,i_m+2,\dots,n)}\notag\\
&\quad \times\int_{T^n}\sigma(dt_1)
\dotsm \sigma(dt_n)\, {:} W(\varkappa_1\cup\varkappa_2\cup\dots\cup \varkappa_m\cup\varkappa_{m+1}
 \cup \zeta(i_1,i_2,\dots,i_m))\notag\\
 &\quad\times \omega(t_1)\dotsm \omega(t_n){:}\, f(t_1)\dotsm f(t_n),\notag
\end{align}
where $\zeta(i_1,i_2,\dots,i_m)$ denotes the element of $G(i_1,i_2,\dots,i_m)$ which has only one element: the set $\{i_1,i_2,\dots,i_m\}$ with mark $+1$.
Therefore, using obvious notations, we get
\begin{align}
&\di_t \la \omega^{\otimes n}, f_1,\otimes\dots\otimes f_n\ra = \sum_{\substack{k=1,\dots,n\\ k\ne2}}\, \sum_{\zeta_1\in NC_{\ge2}(1,2,\dots,k-1)}\notag\\
&\quad\times \int_{T^{k-1}}\sigma(dt_1)\dotsm\sigma(dt_{k-1})W^-(\zeta_1)(t_1,\dots,t_{k-1})f_1(t_1)\dotsm
f_{k-1}(t_{k-1})\notag\\&\quad\times \sum_{m=1}^{n-k-1}\sum_{\substack{ \{i_1,i_2,\dots,i_m\}\subset \{k,k+1,\dots,n\}\\ k=i_1<i_2<\dots<i_m}}\lambda(t)^{m-1} f_{i_1}(t)f_{i_2}(t)\dotsm f_{i_m}(t)\notag\\
&\quad\times\prod_{l=2}^m\sum_{\zeta_l\in NC_{\ge2}(i_{l-1}+1,i_{l-1}+2,\dots,i_l-1)}\int_{T^{i_l-i_{l-1}-1}}
\sigma(dt_{i_{l-1}+1})\sigma(dt_{i_{l-1}+2})\dotsm \sigma (dt_{i_l-1})\notag\\
&\quad\times W(\zeta_l)(t_{i_{l-1}+1},t_{i_{l-1}+2},\dots, t_{i_l-1}) f_{i_{l-1}+1}(t_{i_{l-1}+1})
 f_{i_{l-1}+2}(t_{i_{l-1}+2})\dotsm f_{i_l-1}(t_{i_l-1})\notag\\
 &\quad  \times\left(\sum_{\varkappa\in G(i_m+1,i_m+2,\dots,n)}
 \int_{T^{n-i_m}}\sigma(dt_{i_m+1})\sigma(dt_{i_m+2})\dotsm\sigma(dt_n)\right.\notag\\
 &\quad\left.
 \vphantom{\sum_{\varkappa\in G(i_m+1,i_m+2,\dots,n)}
 \int_{T^{n-i_m}}\sigma(dt_{i_m+1})\sigma(dt_{i_m+2})\dotsm\sigma(dt_n)}
  \times {:}W(\varkappa)\omega(t_{i_m+1})
 \omega(t_{i_m+2})\dotsm\omega(t_n){:}\, f_{i_m+1}(t_{i_m+1})f_{i_m+2}(t_{i_m+2})\dotsm f_n(t_n)
 \right).\label{cfdx}
\end{align}
Since the latter expression in the brackets is equal to
$$ \la \omega^{\otimes(n-i_m)},f_{i_m+1}\otimes f_{i_m+2}\otimes\dots\otimes f_{n}\ra,$$
we conclude  the statement of the theorem from \eqref{cfdx}.
\end{proof}

\begin{center}
{\bf Acknowledgements}\end{center}
The research was partially supported by the International Joint Project grant 2008/R2 of the Royal Society.
 The authors also acknowledge the financial support of the SFB~701 ``Spectral structures and topological methods in mathematics'', Bielefeld University. MB was partially supported by the  
 Polish Ministry of Science and Higher Education, grant N N201 364436.   EL was partially supported by  
  the PTDC/MAT/67965/2006 grant, University of Madeira.

\end{document}